\documentclass{amsart}
\usepackage{amssymb}

\usepackage{amsfonts}

\newtheorem{theorem}{Theorem}
\theoremstyle{plain}

\newtheorem{corollary}{Corollary}

\newtheorem{example}{Example}

\newtheorem{lemma}{Lemma}

\newtheorem{proposition}{Proposition}
\newtheorem{remark}{Remark}

\numberwithin{equation}{section}
\input{tcilatex}

\begin{document}
\title[Inequalities via $s-$convexity and $\log -$convexity]{Inequalities
via $s-$convexity and $\log -$convexity}
\author{Ahmet Ocak Akdemir$^{\bigstar }$}
\address{$^{\bigstar }$A\u{g}r\i\ \.{I}brahim \c{C}e\c{c}en University,
Faculty of Science and Arts, Department of Mathematics, A\u{g}r\i , TURKEY}
\email{ahmetakdemir@agri.edu.tr}
\author{Merve Avc\i\ Ard\i \c{c}$^{\diamondsuit ,\spadesuit }$}
\address{$^{\diamondsuit }$Ad\i yaman University, Faculty of Science and
Arts, Department of Mathematics, Ad\i yaman, TURKEY}
\email{merveavci@ymail.com}
\thanks{$^{\spadesuit }$Corresponding Author}
\author{M. Emin \"{O}zdemir$^{\blacktriangledown }$}
\address{$^{\blacktriangledown }$Atat\"{u}rk University, K. K. Education
Faculty, Department of Mathematics, Erzurum}
\subjclass{26D10, 26A15, 26A16, 26A51.}
\keywords{Convex function, $s-$convex function, $\log -$convex function,
Ostrowski inequality, H\"{o}lder inequality, power-mean inequality.\\
This study was supported by A\u{g}r\i\ \.{I}brahim \c{C}e\c{c}en University
BAP with project number FEF.14.011.}

\begin{abstract}
In this paper, we obtain some new inequalities for functions whose second
derivatives' absolute value is $s-$convex and $\log -$convex. Also, we give
some applications for numerical integration. 
\end{abstract}

\maketitle

\section{INTRODUCTION}

We start with the well-known definition of convex functions: a function $%
f:I\rightarrow 
\mathbb{R}
,$ $\emptyset \neq I\subset 
\mathbb{R}
,$ is said to be convex on $I$ if inequality%
\begin{equation*}
f(tx+(1-t)y)\leq tf(x)+(1-t)f(y)
\end{equation*}%
holds for all $x,y\in I$ and $t\in \left[ 0,1\right] .$

In the paper \cite{hudzik}, authors gave the class of functions which are $%
s- $convex in the second sense by the following way. A function $f:[0,\infty
)\rightarrow 
\mathbb{R}
$ is said to be $s-$convex in the second sence if%
\begin{equation*}
f(tx+(1-t)y)\leq t^{s}f(x)+(1-t)^{s}f(y)
\end{equation*}%
holds for all $x,y\in \lbrack 0,\infty ),t\in \left[ 0,1\right] $ and for
some fixed $s\in (0,1].$ The class of $s-$convex functions in the second
sense is usually denoted with $K_{s}^{2}.$

Besides in \cite{hudzik}, Hudzik and Maligranda proved that if $s\in \left(
0,1\right) $ $f\in K_{s}^{2}$ implies $f([0,\infty ))\subseteq \lbrack
0,\infty ),$ i.e., they proved that all functions from $K_{s}^{2},$ $s\in
\left( 0,1\right) ,$ are nonnegative.

\begin{example}
(\cite{hudzik}) Let $s\in \left( 0,1\right) $ and $a,b,c\in 
\mathbb{R}
.$ We define function $f:[0,\infty )\rightarrow 
\mathbb{R}
$ as 
\begin{equation*}
f(t)=\left\{ 
\begin{array}{cc}
a, & t=0, \\ 
bt^{s}+c, & t>0.%
\end{array}%
\right.
\end{equation*}%
It can be easily checked that

(i) If $b\geq 0$ and $0\leq c\leq a,$ then $f\in K_{s}^{2},$

(ii) If $b>0$ and $c<0,$ then $f\notin K_{s}^{2}.$
\end{example}

Several researchers studied on $s-$convex functions, some of them can be
found in \cite{hudzik}-\cite{zeki}.

Another kind of convexity is $\log -$convexity that is mentioned in \cite{6}
by Niculescu as following.

A positive function $f$ is called $\log -$convex on a real interval $I=\left[
a,b\right] $, if for all $x,y\in \left[ a,b\right] $ and $\lambda \in \left[
0,1\right] $,%
\begin{equation*}
f\left( \lambda x+\left( 1-\lambda \right) y\right) \leq f^{\lambda }\left(
x\right) f^{1-\lambda }\left( y\right) .
\end{equation*}

For recent results for $\log -$convex functions, we refer to readers \cite{1}%
-\cite{18}.

Now, we give a motivated inequality for convex functions:

Let $f:I\subset 
\mathbb{R}
\rightarrow 
\mathbb{R}
$ be a convex function on the interval $I$ of real numbers and $a,b\in I$
with $a<b$. The inequality%
\begin{equation*}
\frac{1}{b-a}\int_{a}^{b}f\left( x\right) dx\leq \frac{1}{2}\left[ f\left( 
\frac{a+b}{2}\right) +\frac{f\left( a\right) +f\left( b\right) }{2}\right]
\end{equation*}%
is known as Bullen's inequality for convex functions \cite{17}, p. 39.

We also consider the following useful inequality:

Let $f:I\subset \left[ 0,\infty \right] \rightarrow 
\mathbb{R}
$ be a differentiable mapping on $I^{\circ }$, the interior of the interval $%
I$, such that $f^{\prime }\in L\left[ a,b\right] $ where $a,b\in I$ with $%
a<b $. If $\left\vert f^{\prime }\left( x\right) \right\vert \leq M$, then
the following inequality holds (see \cite{100}).

\begin{equation}
\left\vert f(x)-\frac{1}{b-a}\int_{a}^{b}f(u)du\right\vert \leq \frac{M}{b-a}%
\left[ \frac{\left( x-a\right) ^{2}+\left( b-x\right) ^{2}}{2}\right]
\label{h.1.1}
\end{equation}

This inequality is well known in the literature as the Ostrowski inequality%
\textit{.}\textbf{\ }

The main aim of this paper is to prove some new integral inequalities for $%
s- $convex and $\log -$convex functions by using the integral identity that
is obtained by Sar\i kaya and Set in \cite{0}. We also give some
applications to our results in numerical integration. Some of our results
are similar to the Ostrowski inequality and for special selections of the
parameters, we proved some new inequalities of Bullen's type.

\section{inequalities for $s-$convex functions}

We need the following Lemma which is obtained by Sar\i kaya and Set in \cite%
{0}, so as to prove our results:

\begin{lemma}
\label{lem 2.1} Let $f:[a,b]\rightarrow 
\mathbb{R}
$ be an absolutely continuous mapping. Denote by $K(x,.):[a,b]\rightarrow 
\mathbb{R}
$ the kernel given by%
\begin{equation*}
K(x,t)=\left\{ 
\begin{array}{ccc}
\frac{\alpha }{\alpha +\beta }\frac{\left( t-a\right) \left( x-t\right) }{x-a%
}, &  & t\in \lbrack a,x] \\ 
&  &  \\ 
-\frac{\beta }{\alpha +\beta }\frac{\left( b-t\right) \left( x-t\right) }{b-x%
}, &  & t\in \lbrack x,b]%
\end{array}%
\right.
\end{equation*}%
where $\alpha ,\beta \in 
\mathbb{R}
$ nonnegative and not both zero, then the identity%
\begin{eqnarray*}
&&\int_{a}^{b}K(x,t)f^{\prime \prime }(t)dt \\
&=&f(x)+\frac{\alpha f(a)+\beta f(b)}{\alpha +\beta }-\frac{2}{\alpha +\beta 
}\left[ \frac{\alpha }{x-a}\int_{a}^{x}f(t)dt+\frac{\beta }{b-x}%
\int_{x}^{b}f(t)dt\right]
\end{eqnarray*}%
holds.
\end{lemma}

\begin{theorem}
\label{teo 2.1} Let $f:[a,b]\rightarrow 
\mathbb{R}
$ be an absolutely continuous mapping such that $f^{\prime \prime }\in
L[a,b].$ If $\left\vert f^{\prime \prime }\right\vert $ is $s-$ convex in
the second sense on $[a,b]$ for some fixed $s\in (0,1],$ then%
\begin{eqnarray*}
&&\left\vert f(x)+\frac{\alpha f(a)+\beta f(b)}{\alpha +\beta }-\frac{2}{%
\alpha +\beta }\left[ \frac{\alpha }{x-a}\int_{a}^{x}f(t)dt+\frac{\beta }{b-x%
}\int_{x}^{b}f(t)dt\right] \right\vert \\
&\leq &\frac{\alpha }{\alpha +\beta }\frac{\left( x-a\right) ^{2}}{\left(
s+2\right) \left( s+3\right) }\left[ \left\vert f^{\prime \prime
}(x)\right\vert +\left\vert f^{\prime \prime }(a)\right\vert \right] \\
&&+\frac{\beta }{\alpha +\beta }\frac{\left( b-x\right) ^{2}}{\left(
s+2\right) \left( s+3\right) }\left[ \left\vert f^{\prime \prime
}(x)\right\vert +\left\vert f^{\prime \prime }(b)\right\vert \right]
\end{eqnarray*}%
holds where $\alpha ,\beta \in 
\mathbb{R}
$ nonnegative and not both zero.
\end{theorem}

\begin{proof}
From Lemma \ref{lem 2.1}, using the property of the modulus and $s-$
convexity of $\left\vert f^{\prime \prime }\right\vert ,$ we can write%
\begin{eqnarray*}
&&\left\vert f(x)+\frac{\alpha f(a)+\beta f(b)}{\alpha +\beta }-\frac{2}{%
\alpha +\beta }\left[ \frac{\alpha }{x-a}\int_{a}^{x}f(t)dt+\frac{\beta }{b-x%
}\int_{x}^{b}f(t)dt\right] \right\vert \\
&\leq &\int_{a}^{b}\left\vert K(x,t)\right\vert \left\vert f^{\prime \prime
}(t)\right\vert dt \\
&\leq &\int_{a}^{x}\frac{\alpha }{\alpha +\beta }\frac{1}{x-a}\left\vert
t-a\right\vert \left\vert x-t\right\vert \left\vert f^{\prime \prime
}(t)\right\vert dt \\
&&+\int_{x}^{b}\frac{\beta }{\alpha +\beta }\frac{1}{b-x}\left\vert
b-t\right\vert \left\vert x-t\right\vert \left\vert f^{\prime \prime
}(t)\right\vert dt \\
&=&\frac{\alpha }{\left( \alpha +\beta \right) \left( x-a\right) }%
\int_{a}^{x}\left( t-a\right) \left( x-t\right) \left\vert f^{\prime \prime
}\left( \frac{t-a}{x-a}x+\frac{x-t}{x-a}a\right) \right\vert dt \\
&&+\frac{\beta }{\left( \alpha +\beta \right) \left( b-x\right) }%
\int_{x}^{b}\left( b-t\right) \left( t-x\right) \left\vert f^{\prime \prime
}\left( \frac{t-x}{b-x}b+\frac{b-t}{b-x}x\right) \right\vert dt \\
&\leq &\frac{\alpha }{\left( \alpha +\beta \right) \left( x-a\right) }%
\int_{a}^{x}\left( t-a\right) \left( x-t\right) \left[ \left( \frac{t-a}{x-a}%
\right) ^{s}\left\vert f^{\prime \prime }(x)\right\vert +\left( \frac{x-t}{%
x-a}\right) ^{s}\left\vert f^{\prime \prime }(a)\right\vert \right] dt \\
&&+\frac{\beta }{\left( \alpha +\beta \right) \left( b-x\right) }%
\int_{x}^{b}\left( b-t\right) \left( t-x\right) \left[ \left( \frac{t-x}{b-x}%
\right) ^{s}\left\vert f^{\prime \prime }(b)\right\vert +\left( \frac{b-t}{%
b-x}\right) ^{s}\left\vert f^{\prime \prime }(x)\right\vert \right] dt \\
&=&\frac{\alpha }{\alpha +\beta }\frac{\left( x-a\right) ^{2}}{\left(
s+2\right) \left( s+3\right) }\left[ \left\vert f^{\prime \prime
}(x)\right\vert +\left\vert f^{\prime \prime }(a)\right\vert \right] +\frac{%
\beta }{\alpha +\beta }\frac{\left( b-x\right) ^{2}}{\left( s+2\right)
\left( s+3\right) }\left[ \left\vert f^{\prime \prime }(x)\right\vert
+\left\vert f^{\prime \prime }(b)\right\vert \right]
\end{eqnarray*}%
where we use the fact that%
\begin{equation*}
\int_{a}^{x}\left( t-a\right) ^{s+1}\left( x-t\right) dt=\int_{a}^{x}\left(
t-a\right) \left( x-t\right) ^{s+1}dt=\frac{\left( x-a\right) ^{s+3}}{\left(
s+2\right) \left( s+3\right) }
\end{equation*}%
and%
\begin{equation*}
\int_{x}^{b}\left( b-t\right) \left( t-x\right) ^{s+1}dt=\int_{x}^{b}\left(
b-t\right) ^{s+1}\left( t-x\right) dt=\frac{\left( b-x\right) ^{s+3}}{\left(
s+2\right) \left( s+3\right) }.
\end{equation*}%
The proof is completed.
\end{proof}

\begin{corollary}
\label{co 2.1} Suppose that all the assumptions of Theorem \ref{teo 2.1} are
satisfied with $\left\vert f^{\prime \prime }\right\vert \leq M.$ Then we
have 
\begin{eqnarray*}
&&\left\vert f(x)+\frac{\alpha f(a)+\beta f(b)}{\alpha +\beta }-\frac{2}{%
\alpha +\beta }\left[ \frac{\alpha }{x-a}\int_{a}^{x}f(t)dt+\frac{\beta }{b-x%
}\int_{x}^{b}f(t)dt\right] \right\vert \\
&\leq &\frac{2M}{\left( s+2\right) \left( s+3\right) }\left[ \frac{\alpha
\left( x-a\right) ^{2}+\beta \left( b-x\right) ^{2}}{\alpha +\beta }\right] .
\end{eqnarray*}
\end{corollary}

\begin{corollary}
\label{co 2.2} In Theorem \ref{teo 2.1}, if we choose $\alpha =\beta =1,$ we
obtain%
\begin{eqnarray*}
&&\left\vert f(x)+\frac{f(a)+f(b)}{2}-\left[ \frac{1}{x-a}\int_{a}^{x}f(t)dt+%
\frac{1}{b-x}\int_{x}^{b}f(t)dt\right] \right\vert \\
&\leq &\frac{\left( x-a\right) ^{2}+\left( b-x\right) ^{2}}{2\left(
s+2\right) \left( s+3\right) }\left\vert f^{\prime \prime }(x)\right\vert +%
\frac{1}{2\left( s+2\right) \left( s+3\right) }\left[ \left( x-a\right)
^{2}\left\vert f^{\prime \prime }(a)\right\vert +\left( b-x\right)
^{2}\left\vert f^{\prime \prime }(b)\right\vert \right] .
\end{eqnarray*}
\end{corollary}

\begin{corollary}
\label{co 2.3} In Theorem \ref{teo 2.1}, if we choose $\alpha =\beta =\frac{1%
}{2}$ and $x=\frac{a+b}{2},$ we obtain the following Bullen type inequality;%
\begin{eqnarray*}
&&\left\vert \frac{1}{2}\left[ f\left( \frac{a+b}{2}\right) +\frac{f(a)+f(b)%
}{2}\right] -\frac{1}{b-a}\int_{a}^{b}f(t)dt\right\vert \\
&\leq &\frac{\left( b-a\right) ^{2}}{8\left( s+2\right) \left( s+3\right) }%
\left[ \left\vert f^{\prime \prime }\left( \frac{a+b}{2}\right) \right\vert +%
\frac{\left\vert f^{\prime \prime }(a)\right\vert +\left\vert f^{\prime
\prime }(b)\right\vert }{2}\right] .
\end{eqnarray*}
\end{corollary}

\begin{theorem}
\label{teo 2.2} Let $f:[a,b]\rightarrow 
\mathbb{R}
$ be an absolutely continuous mapping such that $f^{\prime \prime }\in
L[a,b].$ If $\left\vert f^{\prime \prime }\right\vert ^{q}$ is $s-$ convex
in the second sense on $[a,b]$ for some fixed $s\in (0,1]$ and $q>1$ with $%
\frac{1}{p}+\frac{1}{q}=1,$ then%
\begin{eqnarray*}
&&\left\vert f(x)+\frac{\alpha f(a)+\beta f(b)}{\alpha +\beta }-\frac{2}{%
\alpha +\beta }\left[ \frac{\alpha }{x-a}\int_{a}^{x}f(t)dt+\frac{\beta }{b-x%
}\int_{x}^{b}f(t)dt\right] \right\vert \\
&\leq &\left( \frac{\alpha }{\alpha +\beta }\right) ^{p}\frac{\left(
x-a\right) ^{1+\frac{1}{q}}}{\left( s+1\right) ^{\frac{1}{q}}}\left( \beta
\left( p+1,p+1\right) \right) ^{\frac{1}{p}}\left[ \left\vert f^{\prime
\prime }(x)\right\vert ^{q}+\left\vert f^{\prime \prime }(a)\right\vert ^{q}%
\right] ^{\frac{1}{q}} \\
&&+\left( \frac{\beta }{\alpha +\beta }\right) ^{p}\frac{\left( b-x\right)
^{1+\frac{1}{q}}}{\left( s+1\right) ^{\frac{1}{q}}}\left( \beta \left(
p+1,p+1\right) \right) ^{\frac{1}{p}}\left[ \left\vert f^{\prime \prime
}(b)\right\vert ^{q}+\left\vert f^{\prime \prime }(x)\right\vert ^{q}\right]
^{\frac{1}{q}}
\end{eqnarray*}%
where $\beta \left( x,y\right) =\int_{0}^{1}t^{x-1}\left( 1-t\right)
^{y-1}dt,$ $x,y>0$ is the Euler Beta function, $\alpha ,\beta \in 
\mathbb{R}
$ nonnegative and not both zero.
\end{theorem}

\begin{proof}
From Lemma \ref{lem 2.1}, using the property of the modulus, H\"{o}lder
inequality and $s-$convexity of $\left\vert f^{\prime \prime }\right\vert
^{q},$ we can write%
\begin{eqnarray*}
&&\left\vert f(x)+\frac{\alpha f(a)+\beta f(b)}{\alpha +\beta }-\frac{2}{%
\alpha +\beta }\left[ \frac{\alpha }{x-a}\int_{a}^{x}f(t)dt+\frac{\beta }{b-x%
}\int_{x}^{b}f(t)dt\right] \right\vert \\
&\leq &\left( \int_{a}^{x}\left( \frac{\alpha }{\alpha +\beta }\frac{\left(
t-a\right) \left( x-t\right) }{x-a}\right) ^{p}dt\right) ^{\frac{1}{p}%
}\left( \int_{a}^{x}\left\vert f^{\prime \prime }\left( \frac{t-a}{x-a}x+%
\frac{x-t}{x-a}a\right) \right\vert ^{q}dt\right) ^{\frac{1}{q}} \\
&&+\left( \int_{x}^{b}\left( \frac{\beta }{\alpha +\beta }\frac{\left(
b-t\right) \left( x-t\right) }{b-x}\right) ^{p}dt\right) ^{\frac{1}{p}%
}\left( \int_{x}^{b}\left\vert f^{\prime \prime }\left( \frac{t-x}{b-x}b+%
\frac{b-t}{b-x}x\right) \right\vert ^{q}dt\right) ^{\frac{1}{q}} \\
&\leq &\frac{\alpha }{\alpha +\beta }\left( x-a\right) \left( \int_{a}^{x}%
\frac{\left( t-a\right) ^{p}\left( x-t\right) ^{p}}{\left( x-a\right)
^{p}\left( x-a\right) ^{p}}dt\right) ^{\frac{1}{p}}\left(
\int_{a}^{x}\left\vert f^{\prime \prime }\left( \frac{t-a}{x-a}x+\frac{x-t}{%
x-a}a\right) \right\vert ^{q}dt\right) ^{\frac{1}{q}} \\
&&+\frac{\beta }{\alpha +\beta }\left( b-x\right) \left( \int_{x}^{b}\frac{%
\left( b-t\right) ^{p}\left( x-t\right) ^{p}}{\left( b-x\right) ^{p}\left(
b-x\right) ^{p}}dt\right) ^{\frac{1}{p}}\left( \int_{x}^{b}\left\vert
f^{\prime \prime }\left( \frac{t-x}{b-x}b+\frac{b-t}{b-x}x\right)
\right\vert ^{q}dt\right) ^{\frac{1}{q}} \\
&\leq &\frac{\alpha }{\alpha +\beta }\left( x-a\right) \left( \beta \left(
p+1,p+1\right) \right) ^{\frac{1}{p}}\left[ \int_{a}^{x}\left( \frac{t-a}{x-a%
}\right) ^{s}\left\vert f^{\prime \prime }(x)\right\vert ^{q}+\left( \frac{%
x-t}{x-a}\right) ^{s}\left\vert f^{\prime \prime }(a)\right\vert ^{q}\right]
^{\frac{1}{q}} \\
&&+\frac{\beta }{\alpha +\beta }\left( b-x\right) \left( \beta \left(
p+1,p+1\right) \right) ^{\frac{1}{p}}\left[ \int_{x}^{b}\left( \frac{t-x}{b-x%
}\right) ^{s}\left\vert f^{\prime \prime }(b)\right\vert ^{q}+\left( \frac{%
b-t}{b-x}\right) ^{s}\left\vert f^{\prime \prime }(x)\right\vert ^{q}\right]
^{\frac{1}{q}}.
\end{eqnarray*}%
We get the desired result by making use of the necessary computation.
\end{proof}

\begin{theorem}
\label{teo 2.3} Under the assumptions of Theorem \ref{teo 2.2}, the
following inequality 
\begin{eqnarray*}
&&\left\vert f(x)+\frac{\alpha f(a)+\beta f(b)}{\alpha +\beta }-\frac{2}{%
\alpha +\beta }\left[ \frac{\alpha }{x-a}\int_{a}^{x}f(t)dt+\frac{\beta }{b-x%
}\int_{x}^{b}f(t)dt\right] \right\vert \\
&\leq &\beta \left( p+1,p+1\right) ^{\frac{1}{p}}\left( \left( \frac{\alpha 
}{\alpha +\beta }\right) ^{p}\left( x-a\right) ^{p}+\left( \frac{\beta }{%
\alpha +\beta }\right) ^{p}\left( b-x\right) ^{p}\right) ^{\frac{1}{p}} \\
&&\times \left( \frac{b-a}{s+1}\right) ^{\frac{1}{q}}\left( \left\vert
f^{\prime \prime }(a)\right\vert ^{q}+\left\vert f^{\prime \prime
}(b)\right\vert ^{q}\right) ^{\frac{1}{q}}
\end{eqnarray*}%
holds where $\beta \left( x,y\right) $ is the Euler Beta function.
\end{theorem}

\begin{proof}
From Lemma \ref{lem 2.1}, using the property of the modulus, H\"{o}lder
inequality and $s-$convexity of $\left\vert f^{\prime \prime }\right\vert
^{q},$ we can write%
\begin{eqnarray*}
&&\left\vert f(x)+\frac{\alpha f(a)+\beta f(b)}{\alpha +\beta }-\frac{2}{%
\alpha +\beta }\left[ \frac{\alpha }{x-a}\int_{a}^{x}f(t)dt+\frac{\beta }{b-x%
}\int_{x}^{b}f(t)dt\right] \right\vert \\
&\leq &\left( \int_{a}^{b}\left\vert K(x,t)\right\vert ^{p}dt\right) ^{\frac{%
1}{p}}\left( \int_{a}^{b}\left\vert f^{\prime \prime }(t)\right\vert
^{q}dt\right) ^{\frac{1}{q}} \\
&=&\left( \int_{a}^{x}\left( \frac{\alpha }{\alpha +\beta }\frac{\left(
t-a\right) \left( x-t\right) }{\left( x-a\right) \left( x-a\right) }\left(
x-a\right) \right) ^{p}dt+\int_{x}^{b}\left( \frac{\beta }{\alpha +\beta }%
\frac{\left( b-t\right) \left( x-t\right) }{\left( b-x\right) \left(
b-x\right) }\left( b-x\right) \right) ^{p}dt\right) ^{\frac{1}{p}} \\
&&\times \left( \int_{a}^{b}\left\vert f^{\prime \prime }\left( \frac{t-a}{%
b-a}b+\frac{b-t}{b-a}a\right) \right\vert ^{q}dt\right) ^{\frac{1}{q}}.
\end{eqnarray*}%
We get the desired result by making use of the necessary computation.
\end{proof}

The next result is obtained by using the well-known power-mean integral
ineqaulity:

\begin{theorem}
\label{teo 2.4} Let $f:[a,b]\rightarrow 
\mathbb{R}
$ be an absolutely continuous mapping such that $f^{\prime \prime }\in
L[a,b].$ If $\left\vert f^{\prime \prime }\right\vert ^{q}$ is $s-$ convex
in the second sense on $[a,b]$ for some fixed $s\in (0,1]$ and $q\geq 1,$
then 
\begin{eqnarray*}
&&\left\vert f(x)+\frac{\alpha f(a)+\beta f(b)}{\alpha +\beta }-\frac{2}{%
\alpha +\beta }\left[ \frac{\alpha }{x-a}\int_{a}^{x}f(t)dt+\frac{\beta }{b-x%
}\int_{x}^{b}f(t)dt\right] \right\vert \\
&\leq &\frac{\alpha }{\alpha +\beta }\frac{\left( x-a\right) ^{2}}{6^{1-^{%
\frac{1}{q}}}\left[ \left( s+2\right) \left( s+3\right) \right] ^{\frac{1}{q}%
}}\left[ \left\vert f^{\prime \prime }(x)\right\vert ^{q}+\left\vert
f^{\prime \prime }(a)\right\vert ^{q}\right] ^{\frac{1}{q}} \\
&&+\frac{\beta }{\alpha +\beta }\frac{\left( b-x\right) ^{2}}{6^{1-^{\frac{1%
}{q}}}\left[ \left( s+2\right) \left( s+3\right) \right] ^{\frac{1}{q}}}%
\left[ \left\vert f^{\prime \prime }(b)\right\vert ^{q}+\left\vert f^{\prime
\prime }(x)\right\vert ^{q}\right] ^{\frac{1}{q}}
\end{eqnarray*}%
holds where $\alpha ,\beta \in 
\mathbb{R}
$ nonnegative and not both zero.
\end{theorem}

\begin{proof}
From Lemma \ref{lem 2.1}, using the property of the modulus, power-mean
integral inequality and $s-$convexity of $\left\vert f^{\prime \prime
}\right\vert ^{q},$ we can write%
\begin{eqnarray*}
&&\left\vert f(x)+\frac{\alpha f(a)+\beta f(b)}{\alpha +\beta }-\frac{2}{%
\alpha +\beta }\left[ \frac{\alpha }{x-a}\int_{a}^{x}f(t)dt+\frac{\beta }{b-x%
}\int_{x}^{b}f(t)dt\right] \right\vert  \\
&\leq &\frac{\alpha }{\left( \alpha +\beta \right) \left( x-a\right) }\left(
\int_{a}^{x}\left( t-a\right) \left( x-t\right) dt\right) ^{1-\frac{1}{q}} \\
&&\times \left( \int_{a}^{x}\left( t-a\right) \left( x-t\right) \left(
\left( \frac{t-a}{x-a}\right) ^{s}\left\vert f^{\prime \prime
}(x)\right\vert ^{q}+\left( \frac{x-t}{x-a}\right) ^{s}\left\vert f^{\prime
\prime }(a)\right\vert ^{q}\right) dt\right) ^{\frac{1}{q}} \\
&&+\frac{\beta }{\left( \alpha +\beta \right) \left( b-x\right) }\left(
\int_{x}^{b}\left( b-t\right) \left( t-x\right) dt\right) ^{1-\frac{1}{q}} \\
&&\times \left( \int_{x}^{b}\left( t-b\right) \left( t-x\right) \left(
\left( \frac{t-x}{b-x}\right) ^{s}\left\vert f^{\prime \prime
}(b)\right\vert ^{q}+\left( \frac{b-t}{b-x}\right) ^{s}\left\vert f^{\prime
\prime }(x)\right\vert ^{q}\right) dt\right) ^{\frac{1}{q}} \\
&=&\frac{\alpha }{\left( \alpha +\beta \right) \left( x-a\right) }\left( 
\frac{\left( x-a\right) ^{3}}{6}\right) ^{1-\frac{1}{q}}\left( \frac{\left(
x-a\right) ^{3}}{\left( s+2\right) \left( s+3\right) }\left( \left\vert
f^{\prime \prime }(x)\right\vert ^{q}+\left\vert f^{\prime \prime
}(a)\right\vert ^{q}\right) \right) ^{\frac{1}{q}} \\
&&+\frac{\beta }{\left( \alpha +\beta \right) \left( b-x\right) }\left( 
\frac{\left( b-x\right) ^{3}}{6}\right) ^{1-\frac{1}{q}}\left( \frac{\left(
b-x\right) ^{3}}{\left( s+2\right) \left( s+3\right) }\left( \left\vert
f^{\prime \prime }(b)\right\vert ^{q}+\left\vert f^{\prime \prime
}(x)\right\vert ^{q}\right) \right) ^{\frac{1}{q}}.
\end{eqnarray*}%
The proof is completed.
\end{proof}

\begin{remark}
\label{rem 2.1} In Theorem \ref{teo 2.4}, if we choose $q=1$ Theorem \ref%
{teo 2.4} reduces to Theorem \ref{teo 2.1}.
\end{remark}

\begin{remark}
\label{rem 2.2} If we choose $s=1$ for all the results, we obtain new
results for convex functions.
\end{remark}

\section{inequalities for $\log -$convex functions}

In this section, we will give some results for $\log -$convex functions. For
the simplicity, we will use the following notations:%
\begin{eqnarray*}
\kappa &=&\left( \frac{\left\vert f^{\prime \prime }(x)\right\vert }{%
\left\vert f^{\prime \prime }(a)\right\vert }\right) ^{\frac{1}{x-a}} \\
\tau &=&\left( \frac{\left\vert f^{\prime \prime }(b)\right\vert }{%
\left\vert f^{\prime \prime }(x)\right\vert }\right) ^{\frac{1}{b-x}}.
\end{eqnarray*}

\begin{theorem}
\label{teo 3.1} Let $f:[a,b]\rightarrow 
\mathbb{R}
$ be an absolutely continuous mapping such that $f^{\prime \prime }\in
L[a,b].$ If $\left\vert f^{\prime \prime }\right\vert $ is $\log -$ convex
function on $[a,b]$ and $\kappa \neq 1,\tau \neq 1,$ then%
\begin{eqnarray*}
&&\left\vert f(x)+\frac{\alpha f(a)+\beta f(b)}{\alpha +\beta }-\frac{2}{%
\alpha +\beta }\left[ \frac{\alpha }{x-a}\int_{a}^{x}f(t)dt+\frac{\beta }{b-x%
}\int_{x}^{b}f(t)dt\right] \right\vert \\
&\leq &\frac{\alpha }{\left( \alpha +\beta \right) \left( x-a\right) }\left( 
\frac{\left\vert f^{\prime \prime }(a)\right\vert ^{x}}{\left\vert f^{\prime
\prime }(x)\right\vert ^{a}}\right) ^{\frac{1}{x-a}}\left( \frac{2\left(
\kappa ^{a}-\kappa ^{x}\right) -(a-x)\left( \kappa ^{a}+\kappa ^{x}\right)
\log \kappa }{\log ^{3}\kappa }\right) \\
&&+\frac{\beta }{\left( \alpha +\beta \right) \left( b-x\right) }\left( 
\frac{\left\vert f^{\prime \prime }(x)\right\vert ^{b}}{\left\vert f^{\prime
\prime }(b)\right\vert ^{x}}\right) ^{\frac{1}{b-x}}\left( \frac{2\tau
^{x}-2\tau ^{b}+(b-x)\left( \tau ^{b}+\tau ^{x}\right) \log \tau }{\log
^{3}\tau }\right)
\end{eqnarray*}%
holds where $\kappa \neq 1,\tau \neq 1,$ $\alpha ,\beta \in 
\mathbb{R}
$ nonnegative and not both zero.
\end{theorem}

\begin{proof}
From Lemma \ref{lem 2.1} and by using the $\log -$ convexity of $\left\vert
f^{\prime \prime }\right\vert ,$ we can write%
\begin{eqnarray*}
&&\left\vert f(x)+\frac{\alpha f(a)+\beta f(b)}{\alpha +\beta }-\frac{2}{%
\alpha +\beta }\left[ \frac{\alpha }{x-a}\int_{a}^{x}f(t)dt+\frac{\beta }{b-x%
}\int_{x}^{b}f(t)dt\right] \right\vert \\
&\leq &\int_{a}^{x}\frac{\alpha }{\alpha +\beta }\frac{1}{x-a}\left\vert
t-a\right\vert \left\vert x-t\right\vert \left\vert f^{\prime \prime
}(t)\right\vert dt \\
&&+\int_{x}^{b}\frac{\beta }{\alpha +\beta }\frac{1}{b-x}\left\vert
b-t\right\vert \left\vert x-t\right\vert \left\vert f^{\prime \prime
}(t)\right\vert dt \\
&=&\frac{\alpha }{\left( \alpha +\beta \right) \left( x-a\right) }%
\int_{a}^{x}\left( t-a\right) \left( x-t\right) \left\vert f^{\prime \prime
}\left( \frac{t-a}{x-a}x+\frac{x-t}{x-a}a\right) \right\vert dt \\
&&+\frac{\beta }{\left( \alpha +\beta \right) \left( b-x\right) }%
\int_{x}^{b}\left( b-t\right) \left( t-x\right) \left\vert f^{\prime \prime
}\left( \frac{t-x}{b-x}b+\frac{b-t}{b-x}x\right) \right\vert dt \\
&\leq &\frac{\alpha }{\left( \alpha +\beta \right) \left( x-a\right) }%
\int_{a}^{x}\left( t-a\right) \left( x-t\right) \left[ \left\vert f^{\prime
\prime }(x)\right\vert ^{\frac{t-a}{x-a}}\left\vert f^{\prime \prime
}(a)\right\vert ^{\frac{x-t}{x-a}}\right] dt \\
&&+\frac{\beta }{\left( \alpha +\beta \right) \left( b-x\right) }%
\int_{x}^{b}\left( b-t\right) \left( t-x\right) \left[ \left\vert f^{\prime
\prime }(b)\right\vert ^{\frac{t-x}{b-x}}\left\vert f^{\prime \prime
}(x)\right\vert ^{\frac{b-t}{b-x}}\right] dt \\
&=&\frac{\alpha }{\left( \alpha +\beta \right) \left( x-a\right) }\left( 
\frac{\left\vert f^{\prime \prime }(a)\right\vert ^{x}}{\left\vert f^{\prime
\prime }(x)\right\vert ^{a}}\right) ^{\frac{1}{x-a}}\int_{a}^{x}\left(
t-a\right) \left( x-t\right) \kappa ^{t}dt \\
&&+\frac{\beta }{\left( \alpha +\beta \right) \left( b-x\right) }\left( 
\frac{\left\vert f^{\prime \prime }(x)\right\vert ^{b}}{\left\vert f^{\prime
\prime }(b)\right\vert ^{x}}\right) ^{\frac{1}{b-x}}\int_{x}^{b}\left(
b-t\right) \left( t-x\right) \tau ^{t}dt.
\end{eqnarray*}%
By a simple computation, we get the result.
\end{proof}

\begin{corollary}
\label{co 3.1} In Theorem \ref{teo 3.1}, if we choose $\alpha =\beta =1,$ we
obtain 
\begin{eqnarray*}
&&\left\vert f(x)+\frac{f(a)+f(b)}{2}-\left[ \frac{1}{x-a}\int_{a}^{x}f(t)dt+%
\frac{1}{b-x}\int_{x}^{b}f(t)dt\right] \right\vert \\
&\leq &\frac{1}{2\left( x-a\right) }\left( \frac{\left\vert f^{\prime \prime
}(a)\right\vert ^{x}}{\left\vert f^{\prime \prime }(x)\right\vert ^{a}}%
\right) ^{\frac{1}{x-a}}\left( \frac{2\left( \kappa ^{a}-\kappa ^{x}\right)
-(a-x)\left( \kappa ^{a}+\kappa ^{x}\right) \log \kappa }{\log ^{3}\kappa }%
\right) \\
&&+\frac{1}{2\left( b-x\right) }\left( \frac{\left\vert f^{\prime \prime
}(x)\right\vert ^{b}}{\left\vert f^{\prime \prime }(b)\right\vert ^{x}}%
\right) ^{\frac{1}{b-x}}\left( \frac{2\tau ^{x}-2\tau ^{b}+(b-x)\left( \tau
^{b}+\tau ^{x}\right) \log \tau }{\log ^{3}\tau }\right) .
\end{eqnarray*}
\end{corollary}

\begin{corollary}
\label{co 3.2} In Theorem \ref{teo 3.1}, if we choose $\alpha =\beta =\frac{1%
}{2}$ and $x=\frac{a+b}{2},$ we obtain the following Bullen type inequality;%
\begin{eqnarray*}
&&\left\vert \frac{1}{2}\left[ f\left( \frac{a+b}{2}\right) +\frac{f(a)+f(b)%
}{2}\right] -\frac{1}{b-a}\int_{a}^{b}f(t)dt\right\vert \\
&\leq &\left( \frac{\left\vert f^{\prime \prime }(a)\right\vert ^{\frac{a+b}{%
2}}}{\left\vert f^{\prime \prime }\left( \frac{a+b}{2}\right) \right\vert
^{a}}\right) ^{\frac{2}{b-a}}\left( \frac{\kappa _{1}^{a}-\kappa _{1}^{\frac{%
a+b}{2}}+\left( \frac{b-a}{4}\right) \left( \kappa _{1}^{a}+\kappa _{1}^{%
\frac{a+b}{2}}\right) \log \kappa _{1}}{\left( b-a\right) \log ^{3}\kappa
_{1}}\right) \\
&&+\left( \frac{\left\vert f^{\prime \prime }\left( \frac{a+b}{2}\right)
\right\vert ^{b}}{\left\vert f^{\prime \prime }(b)\right\vert ^{\frac{a+b}{2}%
}}\right) ^{\frac{2}{b-a}}\left( \frac{\tau _{1}^{\frac{a+b}{2}}-\tau
_{1}^{b}+\left( \frac{b-a}{4}\right) \left( \tau _{1}^{b}+\tau _{1}^{\frac{%
a+b}{2}}\right) \log \tau _{1}}{\left( b-a\right) \log ^{3}\tau _{1}}\right)
\end{eqnarray*}%
where 
\begin{eqnarray*}
\kappa _{1} &=&\left( \frac{\left\vert f^{\prime \prime }\left( \frac{a+b}{2}%
\right) \right\vert }{\left\vert f^{\prime \prime }(a)\right\vert }\right) ^{%
\frac{2}{b-a}} \\
\tau _{1} &=&\left( \frac{\left\vert f^{\prime \prime }(b)\right\vert }{%
\left\vert f^{\prime \prime }\left( \frac{a+b}{2}\right) \right\vert }%
\right) ^{\frac{2}{b-a}}.
\end{eqnarray*}
\end{corollary}

\begin{theorem}
\label{teo 3.2} Let $f:[a,b]\rightarrow 
\mathbb{R}
$ be an absolutely continuous mapping such that $f^{\prime \prime }\in
L[a,b].$ If $\left\vert f^{\prime \prime }\right\vert ^{q}$ is $\log -$
convex function on $[a,b]$ and $q>1$ with $\frac{1}{p}+\frac{1}{q}=1,$ then%
\begin{eqnarray*}
&&\left\vert f(x)+\frac{\alpha f(a)+\beta f(b)}{\alpha +\beta }-\frac{2}{%
\alpha +\beta }\left[ \frac{\alpha }{x-a}\int_{a}^{x}f(t)dt+\frac{\beta }{b-x%
}\int_{x}^{b}f(t)dt\right] \right\vert \\
&\leq &\frac{\alpha }{\alpha +\beta }\left( x-a\right) \left( \beta \left(
p+1,p+1\right) \right) ^{\frac{1}{p}}\left( \frac{\left\vert f^{\prime
\prime }(a)\right\vert ^{x}}{\left\vert f^{\prime \prime }(x)\right\vert ^{a}%
}\right) ^{\frac{1}{x-a}}\left( \frac{\kappa ^{\frac{qx}{x-a}}-\kappa ^{%
\frac{qa}{x-a}}}{\log \kappa ^{\frac{q}{x-a}}}\right) ^{\frac{1}{q}} \\
&&+\frac{\beta }{\alpha +\beta }\left( b-x\right) \left( \beta \left(
p+1,p+1\right) \right) ^{\frac{1}{p}}\left( \frac{\left\vert f^{\prime
\prime }(x)\right\vert ^{b}}{\left\vert f^{\prime \prime }(b)\right\vert ^{x}%
}\right) ^{\frac{1}{b-x}}\left( \frac{\tau ^{\frac{qb}{b-x}}-\tau ^{\frac{qx%
}{b-x}}}{\log \tau ^{\frac{q}{b-x}}}\right) ^{\frac{1}{q}}.
\end{eqnarray*}%
where $\beta \left( x,y\right) =\int_{0}^{1}t^{x-1}\left( 1-t\right)
^{y-1}dt,$ $x,y>0$ is the Euler Beta function and $\kappa \neq 1,\tau \neq 1$%
, $\alpha ,\beta \in 
\mathbb{R}
$ nonnegative and not both zero.
\end{theorem}

\begin{proof}
From Lemma \ref{lem 2.1}, by using $\log -$convexity of $\left\vert
f^{\prime \prime }\right\vert ^{q}$ and by applying H\"{o}lder inequality$,$
we get%
\begin{eqnarray*}
&&\left\vert f(x)+\frac{\alpha f(a)+\beta f(b)}{\alpha +\beta }-\frac{2}{%
\alpha +\beta }\left[ \frac{\alpha }{x-a}\int_{a}^{x}f(t)dt+\frac{\beta }{b-x%
}\int_{x}^{b}f(t)dt\right] \right\vert \\
&\leq &\left( \int_{a}^{x}\left( \frac{\alpha }{\alpha +\beta }\frac{\left(
t-a\right) \left( x-t\right) }{x-a}\right) ^{p}dt\right) ^{\frac{1}{p}%
}\left( \int_{a}^{x}\left\vert f^{\prime \prime }\left( \frac{t-a}{x-a}x+%
\frac{x-t}{x-a}a\right) \right\vert ^{q}dt\right) ^{\frac{1}{q}} \\
&&+\left( \int_{x}^{b}\left( \frac{\beta }{\alpha +\beta }\frac{\left(
b-t\right) \left( x-t\right) }{b-x}\right) ^{p}dt\right) ^{\frac{1}{p}%
}\left( \int_{x}^{b}\left\vert f^{\prime \prime }\left( \frac{t-x}{b-x}b+%
\frac{b-t}{b-x}x\right) \right\vert ^{q}dt\right) ^{\frac{1}{q}} \\
&\leq &\frac{\alpha }{\alpha +\beta }\left( x-a\right) \left( \int_{a}^{x}%
\frac{\left( t-a\right) ^{p}\left( x-t\right) ^{p}}{\left( x-a\right)
^{p}\left( x-a\right) ^{p}}dt\right) ^{\frac{1}{p}}\left(
\int_{a}^{x}\left\vert f^{\prime \prime }\left( \frac{t-a}{x-a}x+\frac{x-t}{%
x-a}a\right) \right\vert ^{q}dt\right) ^{\frac{1}{q}} \\
&&+\frac{\beta }{\alpha +\beta }\left( b-x\right) \left( \int_{x}^{b}\frac{%
\left( b-t\right) ^{p}\left( x-t\right) ^{p}}{\left( b-x\right) ^{p}\left(
b-x\right) ^{p}}dt\right) ^{\frac{1}{p}}\left( \int_{x}^{b}\left\vert
f^{\prime \prime }\left( \frac{t-x}{b-x}b+\frac{b-t}{b-x}x\right)
\right\vert ^{q}dt\right) ^{\frac{1}{q}} \\
&\leq &\frac{\alpha }{\alpha +\beta }\left( x-a\right) \left( \beta \left(
p+1,p+1\right) \right) ^{\frac{1}{p}}\left[ \left( \frac{\left\vert
f^{\prime \prime }(a)\right\vert ^{x}}{\left\vert f^{\prime \prime
}(x)\right\vert ^{a}}\right) ^{\frac{q}{x-a}}\int_{a}^{x}\left( \frac{%
\left\vert f^{\prime \prime }(x)\right\vert ^{\frac{q}{x-a}}}{\left\vert
f^{\prime \prime }(a)\right\vert ^{\frac{q}{x-a}}}\right) ^{t}dt\right] ^{%
\frac{1}{q}} \\
&&+\frac{\beta }{\alpha +\beta }\left( b-x\right) \left( \beta \left(
p+1,p+1\right) \right) ^{\frac{1}{p}}\left[ \left( \frac{\left\vert
f^{\prime \prime }(x)\right\vert ^{b}}{\left\vert f^{\prime \prime
}(b)\right\vert ^{x}}\right) ^{\frac{q}{b-x}}\int_{x}^{b}\left( \frac{%
\left\vert f^{\prime \prime }(b)\right\vert ^{\frac{q}{b-x}}}{\left\vert
f^{\prime \prime }(x)\right\vert ^{\frac{q}{b-x}}}\right) ^{t}dt\right] ^{%
\frac{1}{q}}.
\end{eqnarray*}%
By computing the above integrals, we get the desired result.
\end{proof}

\begin{theorem}
\label{teo 3.3} Let $f:[a,b]\rightarrow 
\mathbb{R}
$ be an absolutely continuous mapping such that $f^{\prime \prime }\in
L[a,b].$ If $\left\vert f^{\prime \prime }\right\vert ^{q}$ is $\log -$%
convex function on $[a,b]$ and $q\geq 1,$ then 
\begin{eqnarray*}
&&\left\vert f(x)+\frac{\alpha f(a)+\beta f(b)}{\alpha +\beta }-\frac{2}{%
\alpha +\beta }\left[ \frac{\alpha }{x-a}\int_{a}^{x}f(t)dt+\frac{\beta }{b-x%
}\int_{x}^{b}f(t)dt\right] \right\vert \\
&\leq &\frac{\alpha \left( x-a\right) ^{2-\frac{3}{q}}}{6^{1-\frac{1}{q}%
}\left( \alpha +\beta \right) }\left( \frac{\left\vert f^{\prime \prime
}(a)\right\vert ^{x}}{\left\vert f^{\prime \prime }(x)\right\vert ^{a}}%
\right) ^{\frac{1}{x-a}}\left( \frac{2\left( \kappa ^{qa}-\kappa
^{qx}\right) -q(a-x)\left( \kappa ^{qa}+\kappa ^{qx}\right) \log \kappa }{%
\log ^{3}\kappa ^{q}}\right) ^{\frac{1}{q}} \\
&&+\frac{\beta \left( b-x\right) ^{2-\frac{3}{q}}}{6^{1-\frac{1}{q}}\left(
\alpha +\beta \right) }\left( \frac{\left\vert f^{\prime \prime
}(x)\right\vert ^{b}}{\left\vert f^{\prime \prime }(b)\right\vert ^{x}}%
\right) ^{\frac{1}{b-x}}\left( \frac{2\tau ^{qx}-2\tau ^{qb}+q(b-x)\left(
\tau ^{qb}+\tau ^{qx}\right) \log \tau }{\log ^{3}\tau ^{q}}\right) ^{\frac{1%
}{q}}
\end{eqnarray*}%
holds where $\kappa ^{q}\neq 1,\tau ^{q}\neq 1$, $\alpha ,\beta \in 
\mathbb{R}
$ nonnegative and not both zero.
\end{theorem}

\begin{proof}
From Lemma \ref{lem 2.1}, by using the well-known power-mean integral
inequality and $\log -$convexity of $\left\vert f^{\prime \prime
}\right\vert ^{q},$ we have%
\begin{eqnarray*}
&&\left\vert f(x)+\frac{\alpha f(a)+\beta f(b)}{\alpha +\beta }-\frac{2}{%
\alpha +\beta }\left[ \frac{\alpha }{x-a}\int_{a}^{x}f(t)dt+\frac{\beta }{b-x%
}\int_{x}^{b}f(t)dt\right] \right\vert \\
&\leq &\frac{\alpha }{\left( \alpha +\beta \right) \left( x-a\right) }\left(
\int_{a}^{x}\left( t-a\right) \left( x-t\right) dt\right) ^{1-\frac{1}{q}%
}\left( \int_{a}^{x}\left( t-a\right) \left( x-t\right) \left( \left\vert
f^{\prime \prime }(x)\right\vert ^{q\frac{t-a}{x-a}}\left\vert f^{\prime
\prime }(a)\right\vert ^{q\frac{x-t}{x-a}}\right) dt\right) ^{\frac{1}{q}} \\
&&+\frac{\beta }{\left( \alpha +\beta \right) \left( b-x\right) }\left(
\int_{x}^{b}\left( b-t\right) \left( t-x\right) dt\right) ^{1-\frac{1}{q}%
}\left( \int_{x}^{b}\left( t-b\right) \left( t-x\right) \left( \left\vert
f^{\prime \prime }(b)\right\vert ^{q\frac{t-x}{b-x}}\left\vert f^{\prime
\prime }(x)\right\vert ^{q\frac{b-t}{b-x}}\right) dt\right) ^{\frac{1}{q}} \\
&=&\frac{\alpha }{\left( \alpha +\beta \right) \left( x-a\right) }\left( 
\frac{\left( x-a\right) ^{3}}{6}\right) ^{1-\frac{1}{q}}\left( \frac{%
\left\vert f^{\prime \prime }(a)\right\vert ^{x}}{\left\vert f^{\prime
\prime }(x)\right\vert ^{a}}\right) ^{\frac{1}{x-a}}\left( \frac{2\left(
\kappa ^{qa}-\kappa ^{qx}\right) -q(a-x)\left( \kappa ^{qa}+\kappa
^{qx}\right) \log \kappa }{\log ^{3}\kappa ^{q}}\right) ^{\frac{1}{q}} \\
&&+\frac{\beta }{\left( \alpha +\beta \right) \left( b-x\right) }\left( 
\frac{\left( b-x\right) ^{3}}{6}\right) ^{1-\frac{1}{q}}\left( \frac{%
\left\vert f^{\prime \prime }(x)\right\vert ^{b}}{\left\vert f^{\prime
\prime }(b)\right\vert ^{x}}\right) ^{\frac{1}{b-x}}\left( \frac{2\tau
^{qx}-2\tau ^{qb}+q(b-x)\left( \tau ^{qb}+\tau ^{qx}\right) \log \tau }{\log
^{3}\tau ^{q}}\right) ^{\frac{1}{q}}.
\end{eqnarray*}%
Which completes the proof.
\end{proof}

\begin{remark}
\label{rem 3.1} In Theorem \ref{teo 3.3}, if we choose $q=1$ Theorem \ref%
{teo 3.3} reduces to Theorem \ref{teo 3.1}.
\end{remark}

\begin{corollary}
For the particular selections of the parameters $\alpha ,\beta $ and the
variable $x,$ one can obtain several new inequalities for $\log -$convex
functions, we omit the details.
\end{corollary}

\section{APPLICATIONS FOR NUMERICAL INTEGRATION}

Suppose that $d=\left\{ a=x_{0}<x_{1}<...<x_{n}=b\right\} $ is a partition
of the interval $\left[ a,b\right] ,$ $h_{i}=x_{i+1}-x_{i},$ for $%
i=0,1,2,...,n-1$ and consider the averaged midpoint-trapezoid quadrature
formula%
\begin{equation*}
\int_{a}^{b}f\left( x\right) dx=A_{MT}\left( d,f\right) +R_{MT}\left(
d,f\right) ,
\end{equation*}%
where 
\begin{equation*}
A_{MT}\left( \pi ,f\right) =\frac{1}{4}\sum_{i=0}^{n-1}h_{i}\left[ f\left(
x_{i}\right) +2f\left( \frac{x_{i}+x_{i+1}}{2}\right) +f\left(
x_{i+1}\right) \right]
\end{equation*}%
Here, the term $R_{MT}\left( d,f\right) $ denotes the associated
approximation error. (See \cite{19})

\begin{proposition}
Let $f:[a,b]\rightarrow 
\mathbb{R}
$ be an absolutely continuous mapping such that $f^{\prime \prime }\in
L[a,b].$ If $\left\vert f^{\prime \prime }\right\vert $ is $\log -$ convex
function on $[a,b]$ and $\kappa _{1}\neq 1,\tau _{1}\neq 1,$ then for the
partition $d,$ following inequality holds%
\begin{eqnarray*}
&&\left\vert R_{MT}\left( d,f\right) \right\vert \\
&\leq &\left( \frac{\left\vert f^{\prime \prime }(x_{i})\right\vert ^{\frac{%
x_{i}+x_{i+1}}{2}}}{\left\vert f^{\prime \prime }\left( \frac{x_{i}+x_{i+1}}{%
2}\right) \right\vert ^{x_{i}}}\right) ^{\frac{2}{h_{i}}}\left( \frac{\kappa
_{i}^{x_{i}}-\kappa _{i}^{\frac{x_{i}+x_{i+1}}{2}}+\left( \frac{h_{i}}{4}%
\right) \left( \kappa _{i}^{x_{i}}+\kappa _{i}^{\frac{x_{i}+x_{i+1}}{2}%
}\right) \log \kappa _{i}}{h_{i}\log ^{3}\kappa _{i}}\right) \\
&&+\left( \frac{\left\vert f^{\prime \prime }\left( \frac{x_{i}+x_{i+1}}{2}%
\right) \right\vert ^{x_{i+1}}}{\left\vert f^{\prime \prime
}(x_{i+1})\right\vert ^{\frac{x_{i}+x_{i+1}}{2}}}\right) ^{\frac{2}{h_{i}}%
}\left( \frac{\tau _{i}^{\frac{x_{i}+x_{i+1}}{2}}-\tau _{i}^{x_{i+1}}+\left( 
\frac{h_{i}}{4}\right) \left( \tau _{i}^{x_{i+1}}+\tau _{i}^{\frac{%
x_{i}+x_{i+1}}{2}}\right) \log \tau _{i}}{h_{i}\log ^{3}\tau _{i}}\right) .
\end{eqnarray*}%
where $\kappa _{i}\neq 1,\tau _{i}\neq 1$ and defined as%
\begin{eqnarray*}
\kappa _{i} &=&\left( \frac{\left\vert f^{\prime \prime }\left( \frac{%
x_{i}+x_{i+1}}{2}\right) \right\vert }{\left\vert f^{\prime \prime
}(x_{i})\right\vert }\right) ^{\frac{2}{h_{i}}} \\
\tau _{i} &=&\left( \frac{\left\vert f^{\prime \prime }(x_{i+1})\right\vert 
}{\left\vert f^{\prime \prime }\left( \frac{x_{i}+x_{i+1}}{2}\right)
\right\vert }\right) ^{\frac{2}{h_{i}}}.
\end{eqnarray*}
\end{proposition}

\begin{proof}
By applying Corollary \ref{co 3.2} to the subintervals $\left[ x_{i},x_{i+1}%
\right] $ of $d,$ ($i=0,1,...,n-1)$ and by summation. We obtain the desired
result.
\end{proof}

\begin{proposition}
Let $f:[a,b]\rightarrow 
\mathbb{R}
$ be an absolutely continuous mapping such that $f^{\prime \prime }\in
L[a,b].$ If $\left\vert f^{\prime \prime }\right\vert $ is $s-$ convex in
the second sense on $[a,b]$ for some fixed $s\in (0,1],$ then for partition $%
d$ of $[a,b]$ the following inequality holds:%
\begin{eqnarray*}
&&\left\vert R_{MT}\left( d,f\right) \right\vert \\
&\leq &\frac{h_{i}^{2}}{8\left( s+2\right) \left( s+3\right) }\left[
\left\vert f^{\prime \prime }\left( \frac{x_{i}+x_{i+1}}{2}\right)
\right\vert +\frac{\left\vert f^{\prime \prime }(x_{i})\right\vert
+\left\vert f^{\prime \prime }(x_{i+1})\right\vert }{2}\right] .
\end{eqnarray*}
\end{proposition}

\begin{proof}
By applying Corollary \ref{co 2.3} to the\ subintervals $\left[ x_{i},x_{i+1}%
\right] $ of $d,$ ($i=0,1,...,n-1)$ and by summation. we get the result.
\end{proof}


\begin{thebibliography}{99}
\bibitem{0} M. Z. Sarikaya and E. Set, On new Ostrowski type integral
inequalities, Thai Journal of Mathematics, Vol. 12 (2014), No: 1, 145-154.

\bibitem{1} A.M. Fink, Hadamard's inequality for $\log -$concave functions,
Math. Comput. Modelling 32(5--6) (2000), 625--629.

\bibitem{2} B.G. Pachpatte, A note on integral inequalities involving two $%
\log -$convex functions, Mathematical Inequalities \& Applications, 7(4)
(2004), 511-515.

\bibitem{3} B.G. Pachpatte, A note on Hadamard type integral inequalities
involving several $\log -$convex functions, Tamkang Journal of Mathematics,
36(1) (2005), 43-47.

\bibitem{4} C.E.M. Pearce and J. Pe\v{c}ari\'{c}. Inequalities for
differentiable mappings with application to special means and quadrature
formulae, Appl. Math. Lett., 13(2) 2000, 51-55.

\bibitem{6} C.P. Niculescu, The Hermite--Hadamard inequality for $\log -$%
convex functions, Nonlinear Analysis, 75 (2012), 662--669.

\bibitem{10} G-S. Yang, K-L. Tseng and H-t. Wang, A note on integral
inequalities of Hadamard type for $\log -$convex and $\log -$concave
functions, Taiwanese Journal of Mathematics, 16(2) (2012), 479-496.

\bibitem{17} S.S. Dragomir \& C. Pearce, Selected topics on Hermite-Hadamard
inequalities and applications, Victoria University: RGMIA Monographs, (17)
2000. [http://ajmaa.org/RGMIA/monographs/hermite hadamard.html].

\bibitem{18} S.S. Dragomir, Some Jensen's Type Inequalities for $\log -$%
Convex Functions of Selfadjoint Operators in Hilbert Spaces, Bulletin of the
Malaysian Mathematical Sciences Society, 34(3) (2011), 445-454.

\bibitem{19} S.S. Dragomir, P. Cerone and J. Roumeliotis, A new
generalization of Ostrowski's integral inequality for mappings whose
derivatives are bounded and applications in numerical integration and for
special means, Applied Mathematics Letters, 13 (2000), 19-25.

\bibitem{100} A. Ostrowski,\textit{\ }\"{U}ber die Absolutabweichung einer
differentierbaren Funktion von ihren Integralmittelwert, Comment. Math.
Helv., 10, 226-227, (1938).

\bibitem{hudzik} H. Hudzik, L. Maligranda, Some remarks on $s-$convex
functions, Aequationes Math. 48 (1994) 100-111.

\bibitem{drag} S.S. Dragomir, S. Fitzpatrick, The Hadamard's inequality for $%
s-$convex functions in the second sense. Demonstratio Math. 32 (4) (1999)
687-696.

\bibitem{kirmaci} U.S. K\i rmac\i , M.K. Bakula, M.E. \"{O}zdemir and J. Pe%
\v{c}ari\'{c}, Hadamard-type inequalities for $s-$convex functions, Appl.
Math. Comp., 193 (2007), 26-35.

\bibitem{a1} S. Hussain, M.I. Bhatti and M. Iqbal, Hadamard-type
inequalities for $s-$convex functions, Punjab University, Journal of
Mathematics, 41 (2009) 51-60.

\bibitem{avci} M. Avci, H. Kavurmaci and M.E. \"{O}zdemir, New inequalities
of Hermite--Hadamard type via $s-$convex functions in the second sense with
applications, Appl. Math. and Comput., 217(2011) 5171-5176.

\bibitem{zeki} M.Z. Sarikaya, E. Set and M.E. \"{O}zdemir, On new
inequalities of Simpson's type for $s-$convex functions, Comp. and Math.
with Appl., 60 (2010) 2191-2199.
\end{thebibliography}
\end{document}